\documentclass[11pt,letterpaper]{amsart}

\usepackage{amssymb}
\usepackage{enumerate}

\newtheorem{theorem}{Theorem}
\newtheorem{proposition}[theorem]{Proposition}

\newtheorem{lemma}[theorem]{Lemma}

\newtheorem{atheorem}{Theorem}

\theoremstyle{definition}

\newcommand{\F}{\mathbb{F}}
\renewcommand{\P}{\mathbb{P}}
\newcommand{\e}{\epsilon}

\begin{document}

\title{Exceptional planar polynomials}

\author{Florian Caullery}
\address{Institut de Math{\'e}matiques de Luminy, CNRS-UPR9016, 163 av. de Luminy, case 907, 13288 Marseille Cedex 9, France.}
\email[F. Caullery]{florian.caullery@etu.univ-amu.fr}

\author{Kai-Uwe Schmidt}
\address{Faculty of Mathematics, Otto-von-Guericke University, Universit\"atsplatz~2, 39106 Magdeburg, Germany}
\email[K.-U. Schmidt]{kaiuwe.schmidt@ovgu.de}

\author{Yue Zhou}
\address{Faculty of Mathematics, Otto-von-Guericke University, Universit\"atsplatz~2, 39106 Magdeburg, Germany}
\curraddr{Department of Mathematics and System Sciences, College of Science, National University of Defense Technology, Changsha, China}
\email{yue.zhou.ovgu@gmail.com}

\date{14 February 2014}



\begin{abstract}
Planar functions are special functions from a finite field to itself that give rise to finite projective planes and other combinatorial objects. We consider polynomials over a finite field $K$ that induce planar functions on infinitely many extensions of $K$; we call such polynomials exceptional planar. Exceptional planar monomials have been recently classified. In this paper we establish a partial classification of exceptional planar polynomials. This includes results for the classical planar functions on finite fields of odd characteristic and for the recently proposed planar functions on finite fields of characteristic two.
\end{abstract}

\maketitle

\section{Introduction and Results}

Let $q$ be a prime power. If $q$ is odd, a function $f:\F_q\to\F_q$ is a \emph{planar function} if, for each nonzero $\e\in\F_q$, the function
\begin{equation}
x\mapsto f(x+\e)-f(x)   \label{eqn:def_planar_odd}
\end{equation}
is a permutation on $\F_q$. Such planar functions can be used to construct finite projective planes~\cite{DemOst1968}, relative difference sets~\cite{GanSpe1975}, error-correcting codes~\cite{CarDinYua2005}, and S-boxes in block ciphers~\cite{NybKnu1993}.
\par
If $q$ is even, a function $f:\F_q\to\F_q$ cannot satisfy the above definition of planar functions. This is the motivation to define a function $f:\F_q\to\F_q$ for even $q$ to be \emph{almost perfect nonlinear} (APN) if~\eqref{eqn:def_planar_odd} is a $2$-to-$1$ map. However, there is no apparent link between APN functions and projective planes. Recently, Zhou~\cite{Zho2013} defined a natural analogue of planar functions on finite fields of characteristic two: If $q$ is even, a function $f:\F_q\to\F_q$ is a \emph{planar function} if, for each nonzero $\e\in\F_q$, the function
\[
x\mapsto f(x+\e)+f(x)+\e x
\]
is a permutation on $\F_q$. As shown by Zhou~\cite{Zho2013} and Schmidt and Zhou~\cite{SchZho2013}, such planar functions have similar properties and applications as their counterparts in odd characteristic.
\par
It is well known that every function from $\F_{q^r}$ to itself is induced by a polynomial over $\F_{q^r}$. In this paper, we study polynomials $f\in\F_q[X]$ that induce planar functions on $\F_{q^r}$ for infinitely many $r$; a polynomial $f$ satisfying this property will be called an \emph{exceptional planar polynomial}. Exceptional planar monomials have been completely classified by Leducq~\cite{Led2012} and Zieve~\cite{Zie2013} in odd characteristic and by M\"uller and Zieve~\cite{MulZie2013} in characteristic two. The aim of this paper is to establish a partial classification of exceptional planar polynomials.
\par
We first discuss the case that $q$ is odd, say $q=p^n$, where $p$ is an odd prime. Two polynomials $f,g\in\F_q[X]$ are \emph{extended affine equivalent} (EA-equivalent) if
\[
g(X)=A_1(f(A_2(X)))+A_3(X)
\]
for some polynomials $A_1,A_2,A_3\in\F_q[X]$ with the property that every nonconstant term has degree a power of $p$ and such that $A_1$ and $A_2$ induce permutations on $\F_q$. This equivalence preserves planarity for finite fields of odd characteristic (see~\cite{KyuPot2008} for a discussion on equivalences preserving planarity). Up to EA-equivalence, the only known examples of exceptional planar polynomials on finite fields of odd characteristic are:
\begin{align}
~\!\!\!\!\!\!\!
&\text{$f(X)=X^{p^k+1}$ for some nonnegative integer $k$;}   \label{eqn:exceptional_1}\\[1ex]
&\text{$f(X)=X^{(3^k+1)/2}$ for some odd positive integer $k$, where $p=3$;}   \label{eqn:exceptional_2}\\[1ex]
&\text{$f(X)=X^{10}-uX^6-u^2X^2$ for $u\in\F_{3^n}$ and odd $n$, where $p=3$.}   \label{eqn:exceptional_3}
\end{align}
The polynomial~\eqref{eqn:exceptional_1} is planar on $\F_{p^r}$ for odd $r/\gcd(k,r)$~\cite{CouMat1997}, the polynomial~\eqref{eqn:exceptional_2} is planar on $\F_{3^r}$ for $\gcd(k,r)=1$~\cite{CouMat1997}, and the polynomial~\eqref{eqn:exceptional_3} is planar on $\F_{3^{rn}}$ for odd $r$~\cite{DinYua2006} (see also~\cite{CouMat1997} for the case $u=-1$). As shown in two papers by Leducq~\cite{Led2012} and Zieve~\cite{Zie2013}, up to EA-equivalence, the polynomials~\eqref{eqn:exceptional_1} and~\eqref{eqn:exceptional_2} are the only exceptional planar monomials.
\begin{atheorem}[{Leducq~\cite{Led2012}, Zieve~\cite{Zie2013}}]
\label{thm:monomial_odd}
Let $p$ be an odd prime and let $f\in \F_{p^n}[X]$ be a monic monomial of degree $d$ with $p\nmid d$. If $f$ is exceptional planar, then either~\eqref{eqn:exceptional_1} or~\eqref{eqn:exceptional_2} holds.
\end{atheorem}
\par
A partial classification of exceptional planar polynomials was obtained by Zieve~\cite{Zie2013a}.
\begin{atheorem}[{Zieve~\cite{Zie2013a}}]
\label{thm:polynomial_Zieve}
Let $p$ be an odd prime and let $f\in \F_{p^n}[X]$ be of degree $d$. If $f$ is exceptional planar and $d\not\equiv 0,1 \pmod{p}$, then up to EA-equivalence, either $f(X)=X^2$ or~\eqref{eqn:exceptional_2} holds.
\end{atheorem}
\par
Theorem~\ref{thm:polynomial_Zieve} allows us to restrict ourselves to polynomials over $\F_{p^n}$ whose degrees are congruent to $0$ or $1$ modulo $p$. We prove the following result for the case that the degree is congruent to $1$ modulo $p$.
\begin{theorem}
\label{thm:main_2}
Let $f\in \F_{p^n}[X]$ be monic of degree $d$. If $f$ is exceptional planar and $d\equiv 1 \pmod{p}$, then $f(X)=X^{p^k+1}+h(X)$ for some positive integer $k$, where the degree $e$ of $h$ satisfies $e<p^k+1$ and either $p\mid e$ or $p\mid e-1$.
\end{theorem}
\par
We remark that, except for the trivial case $e=1$, no example is known for which $p\mid e-1$ occurs in Theorem~\ref{thm:main_2}. A nontrivial example for which $p\mid e$ occurs in Theorem~\ref{thm:main_2} is~\eqref{eqn:exceptional_3}.
\par
We now turn to finite fields of characteristic two, in which case the only known examples of exceptional planar polynomials are the polynomials in which the degree of every nonconstant term is a power of two (it is trivial to check that such polynomials are exceptional planar). Indeed, M\"uller and Zieve~\cite{MulZie2013} established the following classification of exceptional planar monomials.
\begin{atheorem}[{M\"uller and Zieve~\cite{MulZie2013}}]
\label{thm:monomial_even}
Let $f\in \F_{2^n}[X]$ be a monomial of degree $d$. If $f$ is exceptional planar, then $d$ is a power of~$2$.
\end{atheorem}
\par
The case that $d$ is odd in Theorem~\ref{thm:monomial_even} was obtained previously by Schmidt and Zhou~\cite{SchZho2013} using different techniques. 
\par
We prove the following partial classification of exceptional planar polynomials.
\begin{theorem}
\label{thm:main_3}
Let $f\in\F_{2^n}[X]$ be of degree $d$. If $f$ is exceptional planar, then either $d\in\{1,2\}$ or $4\mid d$.
\end{theorem}
\par
To prove our main results, we use an approach that has been used to classify polynomials in $\F_{2^n}[X]$ that induce APN functions on infinitely many extensions of $\F_{2^n}$~\cite{Rod2009}, \cite{AubMcGRod2010}, \cite{DelJan2012}, \cite{Cau2013}.
\par
Let $f\in\F_q[X]$ and define the polynomial $F(X,Y,W)$ to be
\begin{equation}
\frac{f(X+W)-f(X)-f(Y+W)+f(Y)}{(X-Y)W}   \label{def:poly_odd}
\end{equation}
when $q$ is odd, and
\begin{equation}
\frac{f(X+W)+f(X)+WX+f(Y+W)+f(Y)+WY}{(X+Y)W}   \label{def:poly_even}
\end{equation}
when $q$ is even. It is a direct consequence of the definition of planar functions that, if $f$ induces a planar function on $\F_{q^r}$, then all $\F_{q^r}$-rational zeros of $F$ satisfy $X=Y$ or $W=0$. The strategy is to show that $F$ has an absolutely irreducible factor over $\F_q$, since then, for all sufficiently large $r$, the polynomial $F$ has many $\F_{q^r}$-rational zeros by the Lang-Weil bound~\cite{LanWei1954}, so that $f$ cannot be planar on $\F_{q^r}$. To do so, we use the key idea of~\cite{AubMcGRod2010} and intersect the projective surface defined by $F$ with a hyperplane and then apply the following result (in which $\overline{\F}_q$ denotes the algebraic closure of $\F_q$).
\begin{lemma}[{Aubry, McGuire, Rodier~\cite[Lemma 2.1]{AubMcGRod2010}}]
\label{lem:AMR}
Let $F$ and $P$ be projective surfaces in $\P^3(\overline{\F}_q)$ defined over $\F_q$. If $F\cap P$ has a reduced absolutely irreducible component defined over $\F_q$, then $F$ has an absolutely irreducible component defined over~$\F_q$.
\end{lemma}
\par
We emphasise that the proof of Theorem~\ref{thm:main_2} relies on intermediate results used to prove Theorem~\ref{thm:monomial_odd}, whereas the proof of Theorem~\ref{thm:main_3} is self-contained, except for the use of the Lang-Weil bound~\cite{LanWei1954} and Lemma~\ref{lem:AMR}.


\section{The odd characteristic case}

In this section we prove Theorem~\ref{thm:main_2}. Let $q$ be an odd prime power, let $f\in\F_q[X]$, and let $F(X,Y,W)$ be the polynomial defined by~\eqref{def:poly_odd}. It will be more convenient to consider the polynomial $F(X,Y,Z-X)$, namely
\begin{equation}
G(X,Y,Z)=\frac{f(X)-f(Y)-f(Z)+f(-X+Y+Z)}{(X-Y)(X-Z)}.   \label{eqn:def_poly_odd}
\end{equation}
Then $f$ induces a planar function on $\F_{q^r}$ if and only if all $\F_{q^r}$-rational zeros of $G$ satisfy $X=Y$ or $X=Z$.
\begin{lemma}
\label{lem:AIC_odd}
Let $q$ be an odd prime power, let $f\in\F_q[X]$ be a polynomial such that $f'$ is nonconstant, and let $G$ be defined by~\eqref{eqn:def_poly_odd}. If $G$ has an absolutely irreducible factor over $\F_q$, then $f$ is not exceptional planar.
\end{lemma}
\begin{proof}
Suppose that $G$ has an absolutely irreducible factor over $\F_q$. Then, by the Lang-Weil bound~\cite{LanWei1954} for the number of rational points in varieties over finite fields, the number of $\F_{q^r}$-rational zeros of $G$ is $q^{2r}+O(q^{3r/2})$, where the implicit constant depends only on the degree of $G$. We claim that $G$ is not divisible by $X-Y$ or $X-Z$. Then $G(X,X,Z)$ and $G(X,Y,X)$ are nonzero polynomials, which have at most $q^r\deg(G)$ zeros in $\F_{q^r}$ (see~\cite[Theorem~6.13]{LidNie1997}, for example). Hence, for all sufficiently large~$r$, the polynomial $G$ has $\F_{q^r}$-rational zeros that do not satisfy $X=Y$ or $X=Z$, and so $f$ does not induce a planar function on $\F_{q^r}$.
\par
It remains to prove the above claim. By symmetry, it is enough to show that $G$ is not divisible by $X-Y$. Suppose for a contradiction that $G$ is divisible by $X-Y$. Then the partial derivative of the numerator of~\eqref{eqn:def_poly_odd} with respect to $Y$, namely $f'(-X+Y+Z)-f'(Y)$, must be divisible by $X-Y$. This forces $f'$ to be a constant polynomial, a contradiction.
\end{proof}
\par
Our main result for finite fields of odd characteristic, Theorem~\ref{thm:main_2}, will follow from Propositions~\ref{pro:main_21} and~\ref{pro:main_22}, to be stated and proved below. Before we proceed, we introduce some notation that will be used throughout the remainder of this section. Let $d$ be the degree $f\in\F_q[X]$. Write $f(X)=\sum_{j=0}^da_jX^j$, where $a_d\ne 0$. Defining 
\begin{equation}
\phi_j(X,Y,Z)=\frac{X^j-Y^j-Z^j+(-X+Y+Z)^j}{(X-Y)(X-Z)},   \label{eqn:def_phi}
\end{equation}
we have
\begin{equation}
G(X,Y,Z)=\sum_{j=2}^da_j\phi_j(X,Y,Z).   \label{eqn:G_phi}
\end{equation}
since $\phi_0=\phi_1=0$. We shall also work with the homogeneous polynomial
\[
\widetilde{G}(X,Y,Z,T)=\sum_{j=2}^da_j\phi_j(X,Y,Z)\,T^{d-j}.
\]
\begin{proposition}
\label{pro:main_21}
Let $p$ be an odd prime and let $f\in \F_{p^n}[X]$ be of degree $d$. If $f$ is exceptional planar and $d\equiv 1 \pmod{p}$, then $d=p^k+1$ for some positive integer $k$.
\end{proposition}
\par
To prove Proposition~\ref{pro:main_21}, we require the following lemma.
\begin{lemma}
\label{lem:odd_degree}
Let $p$ be an odd prime and let $f\in \F_{p^n}[X]$ be of degree $d$. If $f$ is exceptional planar and $p\nmid d$, then $d$ is even.
\end{lemma}
\begin{proof}
Suppose for a contradiction that $f$ is exceptional planar and $p\nmid d$ and $d$ is odd. By the definition of a planar function, the degree of $f$ must be at least $2$, so that $f'$ is not constant. The intersection of $\widetilde{G}$ with the hyperplane $T=0$ is defined by the polynomial
\[
\widetilde{G}(X,Y,Z,0)=a_d\phi_d(X,Y,Z).
\]
Since $d$ is odd, $Y+Z$ divides $\phi_d$. By taking the partial derivative of $X^d-Y^d-Z^d+(-X+Y+Z)^d$ with respect to $Y$, we see that $(Y+Z)^2$ does not divide $\phi_d$. Therefore $Y+Z$ is a reduced absolutely irreducible component of $\phi_d$ and hence, by Lemma~\ref{lem:AMR}, $\widetilde{G}$ (and so also $G$ itself) has an absolutely irreducible factor over $\F_{p^n}$. Therefore, by Lemma~\ref{lem:AIC_odd}, the polynomial $f$ is not exceptional planar, a contradiction.
\end{proof}
\par
We now prove Proposition~\ref{pro:main_21}.
\begin{proof}[Proof of Proposition~\ref{pro:main_21}]
Suppose that $f$ is exceptional planar and $d\equiv 1 \pmod{p}$. We show that this is impossible unless $d$ is of the form $p^k+1$. 
\par
If $d=p^i(p^i-1)+1$ for some nonnegative integer~$i$, then $d$ is odd and $f$ is not exceptional planar by Lemma~\ref{lem:odd_degree}, so assume that $d$ is not of this form. In particular, $f'$ is not constant. The intersection of $\widetilde{G}$ with the hyperplane $T=0$ is defined by the polynomial
\[
\widetilde{G}(X,Y,Z,0)=\phi_d(X,Y,Z).
\]
Since $d\equiv 1\pmod p$ and $d$ is not of the form $p^k(p^k-1)+1$, the polynomial
\[
\phi_d(U+W,U,V+W)=\frac{(U+W)^d-U^d-(V+W)^d+V^d}{(U-V)W}
\]
has an absolutely irreducible factor of $\F_p$ provided that $d$ is not of the form $p^k+1$, as shown by Leducq~\cite{Led2012}. Furthermore, Leducq~\cite{Led2012} showed that the number of singular points of $\phi_d(U+W,U,V+W)$ is finite. Hence the variety defined by $\phi_d$ and all of its partial derivatives has dimension $0$, which implies that $\phi_d$ has no multiple component. Therefore $\phi_d$ has a reduced absolutely irreducible factor over $\F_p$ and so, by Lemmas~\ref{lem:AMR} and~\ref{lem:AIC_odd}, the polynomial $f$ is not exceptional planar, a contradiction.
\end{proof}
\par
\begin{proposition}
\label{pro:main_22}
Let $p$ be an odd prime and let $f\in \F_{p^n}[X]$ be of the form $f(X)=X^{p^k+1}+h(X)$ for some positive integer $k$, where the degree $e$ of $h$ satisfies $e<p^k+1$. If $f$ is exceptional planar, then either $p\mid e$ or $p\mid e-1$.
\end{proposition}
\par
In order to prove Proposition~\ref{pro:main_22}, we prove two lemmas on the polynomials $\phi_j$, defined by~\eqref{eqn:def_phi}.
\begin{lemma}
\label{lem:phi_j}
Let $p$ be a prime and let $\phi_j\in\F_p[X,Y,Z]$ be defined by~\eqref{eqn:def_phi}.
\begin{enumerate}[(i)]
\item We have
\[
\phi_j(X,X,Z)=j\,\frac{X^{j-1}-Z^{j-1}}{X-Z}.
\]
\item If $p\nmid j$ and $p\nmid j-1$, then $\phi_j(X,X,Z)$ is not divisible $X-Z$ and $\phi_j(X,X,Z)$ and $\phi_{p^k+1}(X,X,Z)$ are coprime.
\end{enumerate}
\end{lemma}
\begin{proof}
We may write 
\begin{align*}
(X-Z)\phi_j(X,Y,Z)&=\frac{X^j-Y^j}{X-Y}-\frac{(-X+Y+Z)^j-Z^j}{(-X+Y+Z)-Z}\\
                  &=\sum_{i=0}^{j-1}X^i\,Y^{j-i-1}-\sum_{i=0}^{j-1}(-X+Y+Z)^i\,Z^{j-i-1},
\end{align*}
from which (i) follows. If $p\nmid j$, then $\phi_j(X,X,Z)$ is not the zero polynomial. If, in addition, $p\nmid j-1$, then $\phi_j(X,X,Z)$ splits into linear factors different from $X-Z$. From (i) we have $\phi_{p^k+1}(X,X,Z)=(p^k+1)(X-Z)^{p^k-1}$. Hence, if $p\nmid j$ and $p\nmid j-1$, then $\phi_j(X,X,Z)$ and $\phi_{p^k+1}(X,X,Z)$ are coprime. This proves (ii). 
\end{proof}
\par
\begin{lemma}
\label{lem:phi_pk1_square-free}
Let $p$ be a prime and let $\phi_j\in\F_p[X,Y,Z]$ be defined by~\eqref{eqn:def_phi}. Then $\phi_{p^k+1}$ is square-free.
\end{lemma}
\begin{proof}
Write
\[
\psi(X,Y,Z)=X^{p^k+1}-Y^{p^k+1}-Z^{p^k+1}+(-X+Y+Z)^{p^k+1}.
\]
Then $\phi_{p^k+1}$ divides $\psi$. We show that $\psi$ is square-free, for which it is sufficient to show that all of the following conditions are satisfied:
\begin{itemize}
\item $\gcd(\psi,\partial \psi/\partial Y)\in\F_p[X,Z]$, 
\item $\gcd(\psi,\partial \psi/\partial Z)\in\F_p[X,Y]$,
\item $X\nmid \psi$.
\end{itemize}
These conditions are readily verified. 
\end{proof}
\par
We now prove Proposition~\ref{pro:main_22}, using an idea of Delgado and Janwa~\cite{DelJan2012}.
\begin{proof}[Proof of Proposition~\ref{pro:main_22}]
Suppose that $f$ is exceptional planar. Then $G$, defined in~\eqref{eqn:def_poly_odd}, is not absolutely irreducible by Lemma~\ref{lem:AIC_odd}. Suppose further that $p\nmid e$ and $p\nmid e-1$. We show that this leads to a contradiction.
\par
We may write
\begin{equation}
G(X,Y,Z)=(P_s+P_{s-1}+\cdots+P_0)(Q_t+Q_{t-1}+\cdots+Q_0),   \label{eqn:G_PQ}
\end{equation}
where $P_i$ and $Q_i$ are zero or homogeneous polynomials of degree $i$, defined over the algebraic closure of $\F_{p^n}$, and $P_sQ_t$ is nonzero. Since $G$ is not absolutely irreducible, we may also assume that $s,t>0$. Write
\[
f(X)=\sum_{j=0}^{p^k+1}a_jX^j,
\]
where $a_{p^k+1}=1$ and recall from~\eqref{eqn:G_phi} that
\begin{equation}
G(X,Y,Z)=\sum_{j=2}^{p^k+1}a_j\phi_j(X,Y,Z),   \label{eqn:G_phi_2}
\end{equation}
where the $\phi_j$'s are defined in~\eqref{eqn:def_phi}. Notice that the degree of $\phi_j$ is $j-2$ and thus, since $a_{p^k+1}=1$, the degree of $G$ is $p^k-1$. Hence $s+t=p^k-1$ by~\eqref{eqn:G_PQ}. From~\eqref{eqn:G_PQ} and~\eqref{eqn:G_phi_2} we find that
\begin{equation}
P_sQ_t=\phi_{p^k+1}.   \label{eqn:PQ_phi_pk1}
\end{equation}
Therefore, by Lemma~\ref{lem:phi_pk1_square-free}, $P_s$ and $Q_t$ are coprime. From~\eqref{eqn:G_PQ} and~\eqref{eqn:G_phi_2} we also find that
\[
P_sQ_{t-1}+P_{s-1}Q_t=a_{p^k}\phi_{p^k}=0
\]
since $\phi_{p^k}=0$. Hence $P_s$ divides $P_{s-1}Q_t$ and so $P_s$ divides $P_{s-1}$, which by a degree argument implies that $P_{s-1}=0$. Likewise, we see that $Q_{t-1}=0$. Now, by the assumed form of $f$, we have
\begin{equation}
a_j=0\quad\text{for each $j\in\{p^k,p^k-1,\dots,e+1\}$}.   \label{eqn:a_j_zero}
\end{equation}
Since $P_{s-1}=Q_{t-1}=0$, we have from~\eqref{eqn:G_PQ} and~\eqref{eqn:G_phi_2}
\[
P_sQ_{t-2}+P_{s-2}Q_t=a_{p^k-1}\phi_{p^k-1}.
\]
If $p^k-1\ge e+1$, the right hand side equals zero by~\eqref{eqn:a_j_zero} and, by an argument similar to that used above, we conclude that $P_{s-2}=Q_{t-2}=0$. We can continue in this way to show that
\[
P_{t-1}=\cdots=P_{e-t-1}=Q_{s-1}=\cdots=Q_{e-s-1}=0.
\]
Hence, by invoking~\eqref{eqn:G_PQ} and~\eqref{eqn:G_phi_2} again, we have
\begin{equation}
P_sQ_{e-s-2}+P_{e-t-2}Q_t=a_e\phi_e.   \label{eqn:PQ_phi_e}
\end{equation}
If $Q_{e-s-2}=0$, then $Q_t$ divides $\phi_e$ and also $\phi_{p^k+1}$ by~\eqref{eqn:PQ_phi_pk1}. This contradicts Lemma~\ref{lem:phi_j} (ii). Likewise, we get a contradiction if $P_{e-t-2}=0$. Hence we may assume that $P_{e-t-2}$ and $Q_{e-s-2}$ are both nonzero. From~\eqref{eqn:PQ_phi_pk1} and Lemma~\ref{lem:phi_j}~(i) we find that
\begin{equation}
P_s(X,X,Z)Q_t(X,X,Z)=\phi_{p^k+1}(X,X,Z)=(X-Z)^{p^k-1}.   \label{eqn:PQ_XXZ}
\end{equation}
Hence $X-Z$ divides $P_s(X,X,Z)$ and $Q_t(X,X,Z)$ and thus $X-Z$ also divides $\phi_e(X,X,Z)$ by~\eqref{eqn:PQ_phi_e}. From~\eqref{eqn:PQ_XXZ} we then see that $\phi_e(X,X,Z)$ and $\phi_{p^k+1}(X,X,Z)$ share the factor $X-Z$, which contradicts Lemma~\ref{lem:phi_j} (ii).
\end{proof}


\section{The even characteristic case}

In this section we prove Theorem~\ref{thm:main_3}. Let $q$ be a power of $2$, let $f\in\F_q[X]$, and let $F(X,Y,W)$ be the polynomial defined by~\eqref{def:poly_even}. We consider the polynomial $F(X,Y,Z+X)$, namely
\begin{equation}
H(X,Y,Z)=\frac{f(X)+f(Y)+f(Z)+f(X+Y+Z)}{(X+Y)(X+Z)}+1.   \label{eqn:def_poly_even}
\end{equation}
Then $f$ induces a planar function on $\F_{q^r}$ if and only if all $\F_{q^r}$-rational zeros of $H$ satisfy $X=Y$ or $X=Z$.
\par
The following lemma is our counterpart of Lemma~\ref{lem:AIC_odd} in even characteristic.
\begin{lemma}
\label{lem:AIC_even}
Let $q$ be a power of $2$, let $f\in\F_q[X]$, and let $H$ be defined by~\eqref{eqn:def_poly_even}. If $H$ has an absolutely irreducible factor over $\F_q$, then $f$ is not exceptional planar.
\end{lemma}
\begin{proof}
Suppose that $H$ has an absolutely irreducible factor over $\F_q$. Then, as in the proof of Lemma~\ref{lem:AIC_odd}, the number of $\F_{q^r}$-rational zeros of $H$ is $q^{2r}+O(q^{3r/2})$. We show below that $H$ is not divisible by $X+Y$ or $X+Z$, which implies that, for all sufficiently large~$r$, the polynomial $H$ has $\F_{q^r}$-rational zeros that do not satisfy $X=Y$ or $X=Z$, and so $f$ does not induce a planar function on $\F_{q^r}$.
\par
By symmetry, it is enough to show that $H$ is not divisible by $X+Y$. Suppose for a contradiction that $H$ is divisible by $X+Y$. Then the partial derivative of $(X+Y)(X+Z)H(X,Y,Z)$ with respect to $Y$, namely
\[
f'(Y)+f'(X+Y+Z)+X+Z,
\]
must be divisible by $X+Y$. This forces $f'(X)=X+c$ for some $c\in\F_q$, which is absurd since $q$ is even.
\end{proof}
\par
We now prove Theorem~\ref{thm:main_3}.
\begin{proof}[Proof of Theorem~\ref{thm:main_3}]
Write $f(X)=\sum_{j=0}^da_jX^j$, where $a_d\ne 0$. Defining
\[
\phi_j(X,Y,Z)=\frac{X^j+Y^j+Z^j+(X+Y+Z)^j}{(X+Y)(X+Z)},
\]
the polynomial $H$, defined in~\eqref{eqn:def_poly_even}, can be written as
\[
H(X,Y,Z)=1+\sum_{j=3}^da_j\phi_j(X,Y,Z)
\]
since $\phi_0=\phi_1=\phi_2=0$. Consider the homogeneous polynomial
\[
\widetilde{H}(X,Y,Z,T)=T^{d-2}+\sum_{j=3}^da_j\phi_j(X,Y,Z)\,T^{d-j}.
\]
The intersection of the projective surface defined by $\widetilde{H}$ with the hyperplane $T=0$ is defined by the polynomial
\[
\widetilde{H}(X,Y,Z,0)=a_d\phi_d(X,Y,Z).
\]
Now suppose for a contradiction that $f$ is exceptional planar and $4\nmid d$, but $d\not\in\{1,2\}$. We show that $\phi_d$ has a reduced absolutely irreducible component, which by Lemmas~\ref{lem:AMR} and~\ref{lem:AIC_even} implies that $f$ is not exceptional planar, a contradiction.
\par
First suppose that $d$ is odd and $d\ne 1$. Then $Y+Z$ divides $\phi_d$. By taking the partial derivative of $X^d+Y^d+Z^d+(X+Y+Z)^d$ with respect to $Y$, we see that $(Y+Z)^2$ does not divide $\phi_d$. Therefore $Y+Z$ is a reduced absolutely irreducible component of $\phi_d$, as required.
\par
Now suppose that $d\equiv 2\pmod 4$ and $d\ne 2$. Write $d=2e$, so that $e$ is odd and $e\ne 1$. It is readily verified that
\[
\phi_d=\phi_e^2\cdot(X+Y)(X+Z).
\]
Hence $X+Y$ divides $\phi_d$. By taking the partial derivative of $X^e+Y^e+Z^e+(X+Y+Z)^e$ with respect to $Y$, we find that $X+Y$ does not divide $\phi_e$ and so $(X+Y)^2$ does not divide $\phi_d$. Hence $X+Y$ is a reduced absolutely irreducible component of $\phi_d$, as required.
\end{proof}



\providecommand{\bysame}{\leavevmode\hbox to3em{\hrulefill}\thinspace}
\providecommand{\MR}{\relax\ifhmode\unskip\space\fi MR }
\providecommand{\MRhref}[2]{%
  \href{http://www.ams.org/mathscinet-getitem?mr=#1}{#2}
}
\providecommand{\href}[2]{#2}

\end{document}